\documentclass{article}
\usepackage[utf8]{inputenc}
\usepackage{amsmath, amsthm, amssymb, tikz}

\newtheorem{thm}{Theorem}[section]

\newtheorem{lem}[thm]{Lemma}
\newtheorem{cor}[thm]{Corollary}
\newtheorem{prop}[thm]{Proposition}

\theoremstyle{definition}
\newtheorem{dfn}[thm]{Definition}

\theoremstyle{remark}

\newcommand{\Z}{\mathbb{Z}}

\newcommand{\baire}{\mathcal{N}}

\newcommand{\bb}[1]{\mathbb{{#1}}}
\newcommand{\inv}{^{-1}}
\newcommand{\fr}{^\smallfrown} %concatenation sign
\newcommand{\ip}[1]{\langle {#1} \rangle} %inner product
\newcommand{\sm}{\smallsetminus}
 
\newcommand{\bp}[2]{\mathbf{\Pi}_{#2}^{#1}} %boldface hierarchy pi
\newcommand{\bd}[2]{\mathbf{\Delta}_{#2}^{#1}} % and delta
\newcommand{\Fr}{\mathrm{Fr}}

\newcommand{\dom}{\operatorname{dom}}

\newcommand{\aut}{\operatorname{Aut}}
\newcommand{\cay}{\operatorname{Cay}}

\newcommand{\res}{\upharpoonright}

\newcommand{\lra}{\Leftrightarrow}

\newcommand{\s}[1]{\mathcal{{#1}}} %caligraphy

\title{$\Delta^1_1$ Effectivization in Borel Combinatorics}
\author{Riley Thornton }

\begin{document}

\maketitle

\begin{abstract}
    We develop a flexible method for showing that Borel witnesses to some combinatorial property of $\Delta^1_1$ objects yield $\Delta^1_1$ witnesses. We use a modification the Gandy--Harrington forcing method of proving dichotomies, and we can recover the complexity consequences of many known dichotomies with short and simple proofs. Using our methods, we give a simplified proof that smooth $\Delta^1_1$ equivalence relations are $\Delta^1_1$-reducible to equality;  we prove effective versions of the Lusin--Novikov and Feldman--Moore theorems; we prove new effectivization results related to dichotomy theorems due to Hjorth and Miller (originally proven using ``forceless, ineffective, and powerless" methods); and we prove a new upper bound on the complexity of the set of Schreier graphs for $\Z^2$ actions. We also prove an equivariant version of the $G_0$ dichotomy that implies some of these new results and a dichotomy for graphs induced by Borel actions of $\Z^2$.
\end{abstract}

\section{Introduction}

Recent research has focused on the projective complexity of various Borel combinatorial properties of Borel graphs, equivalence relations, etc. For instance, the main result of \cite{TV} is that the set of (codes for) Borel 3-colorable Borel graphs is $\mathbf{\Sigma}^1_2$ complete. Such a complexity lower bound rules out dichotomy theorems like the $G_0$ dichotomy characterizing countably Borel colorable graphs \cite{KST}. It also implies the existence of $\Delta^1_1$ graphs which are Borel 3-colorable but not $\Delta^1_1$ 3-colorable.  In this note, we explain a modification of the Gandy--Harrington forcing arguments for dichotomy theorems which yields effectivization results. That is, we show how Borel witnesses to combinatorial properties of $\Delta^1_1$ objects imply the existence of $\Delta^1_1$ witnesses. This kind of effectivization implies a strong upper bound on projective complexity, and can be construed as a sort of weak dichotomy theorem. Our methods give short and simple proofs of these weak dichotomies even for properties where an actual dichotomy is unwieldy, and they allow us to make a fairly sweeping generalization of many known results. 

The basic tool that we use is the observation that, for $\Phi$ among a large class of properties including what we will call independence properties, if $B$ is a Borel set, $A$ is a Gandy--Harrington condition, $\Phi(B)$, and $A\Vdash \dot x\in B$, then $\Phi(A)$. So, for instance, if a space can be covered with countably many Borel independent sets, then in any nonempty $\Sigma^1_1$ set we can use the forcing relation to find an independent set, so by a reflection argument we have a cover by independent $\Delta^1_1$ sets. This is made precise in Lemma \ref{theft} and Theorem \ref{general}. %hmmmmmm

In the next section we give three illustrative examples, including a short proof that smooth $\Delta^1_1$ equivalence relations are $\Delta^1_1$-reducible to equality. In Section \ref{general section} we define a large class of properties $\Phi$ where this method works and prove some general results. In Section \ref{applications} we give several applications. We prove effective versions of the Lusin--Novikov and Feldman--Moore theorems ({Theorem \ref{effectivelnfm}}). We show that any $\Delta^1_1$ graph generated by a single Borel function, a countable family of Borel functions, or a Borel free action of $\Z^2$ must be generated by a $\Delta^1_1$ function, countable family of functions, or $\Z^2$-action (Theorems \ref{1function}, \ref{manyfunctions}, and \ref{actions} respectively). We show that if $G$ is locally countable, $\Delta^1_1$, and admits a Borel end selection, then it admits a $\Delta^1_1$ end selection (Theorem \ref{selection}). And, we prove effectivization for Borel local colorings in the sense of Miller's $(\bb G_0, \bb H_0)$ dichotomy (Theorem \ref{goho}). In the last section we prove a handful of new dichotomy theorems related to some of the effectivization results in section 4. In particular, we prove an equivariant version of the $G_0$ dichotomy which implies the results about generating graphs with functions (Theorem \ref{equivariant go}). And, we prove a dichotomy for Schreier graphs of Borel actions of $\Z^2$ which generalizes Miller's characterization of undirectable forests of lines (Theorem \ref{axndichotomy}).

% \subsection{Acknowledgements} Thank you to Andrew Marks for helpful comments on a draft of this paper. The author was supported by the NSF grant DMS-1764174

\subsection{Conventions and Notation}

Throughout, graphs are simple and undirected. Formally, a graph on a vertex set $X$ is a symmetric, irreflexive subset of $X^2$. An edge in a graph is some ordered pair in the graph (so our edges are directed, and our graphs contain both possible directions).  

For an edge $e$, we typically write $e_0$ for the tail and $e_1$ for the head of $e$. And, we write $-e$ for the edge with the opposite direction, i.e.~$-(e_0,e_1)=(e_1,e_0)$. For a set of edges $A$, let $-A=\{-e: e\in A\}.$

We write $\baire$ for Baire space, i.e.~$\omega^\omega$ with the product topology. The map $\pi_i:\baire^n\rightarrow \baire$ is projection onto the $i^{th}$ coordinate. A box is a subset of $\baire^n$ of the form $A_1\times...\times A_n$ for some $A_1,...,A_n\subseteq \baire$.

We will work with Gandy--Harrington forcing on $\baire^n$, i.e.~the poset of non-empty $\Sigma^1_1$ subsets of $\baire^n$ ordered by inclusion. We write for $\bb P_n$ this forcing, and we write $\dot x$ for the standard name for the real coded by a $\bb P_1$-generic filter, $(\dot x,\dot y)$ for the name for the reals coded by a $\bb P_2$-generic, and $\dot {\mathbf x}=(\dot x_1,...,\dot x_n)$ for the $\bb P_n$-generic reals.

For a Borel set $B$, when we write $\Vdash \dot x\in B$ or consider $B$ in some generic extension, we mean the set in coded in the extension by some ground model code for $B$. In particular, $\Phi(B)$ is absolute whenever $\Phi$ is $\bp11$ on $\bd11$.

When we say a $\Delta^1_1$ sequence of $\Delta^1_1$ sets $\ip{A_\alpha: \alpha\in \beta}$ (for a computable ordinal $\beta$), we mean that there is a $\Delta^1_1$ order $\prec$ on $\omega$ and $\Delta^1_1$ set of pairs $(c_\alpha, i_\alpha)$ so that $i_\alpha\in\omega$ has order type $\alpha$ in $\prec$ and  $c_\alpha$ is a code for $A_\alpha$. A family of Borel sets is $\Pi^1_1$ in the codes if the set of Borel codes for Borel sets in this family is $\Pi^1_1.$ Our references for effective descriptive set theory are Marks's lecture notes \cite{edst} and the paper by Harrington, Marker, and Shelah \cite{HMS}. 

When $s_1,s_2,s_3,...$ are ordinal length sequences, say $s_i\in A^{\alpha_i}$ for some $A$, then their concatenation, $s=s_1\fr s_2\fr s_3\fr ...$ is the squences of length $\sum_i \alpha_i$ defined by
\[s(\beta)=s_i(\gamma):\lra \beta=(\sum_{j<i}\alpha_j )+\gamma\]

\subsection{Acknowledgements} Thanks to Andrew Marks for helpful comments on earlier drafts of this paper. The author was supported by the NSF grant DMS-1764174.

\section{First examples}

We begin with three examples. First we reprove the effectivization consequences of the $G_0$ dichotomy:

\begin{thm}
If $g$ is a code for a Borel graph $G$ on $\baire$, and $G$ admits a Borel countable coloring, then $G$ admits a $\Delta^1_1(g)$ countable coloring.
\end{thm}
\begin{proof}
By relativization, we may assume $g\in \Delta^1_1.$ We first show that $G$ admits a $\Delta^1_1$ coloring if and only if $\baire$ is a union of $\Delta^1_1$ $G$-independent sets. If $f:\baire\rightarrow \omega$ is a coloring, then $\baire=\bigcup_i f\inv(i)$ and each of the fibers is $G$-independent and $\Delta^1_1$. For the converse, suppose every $x$ is in some $\Delta^1_1$ independent set. Then \[\{(x,i):i\mbox{ codes }A, \; x\in A,\; A\mbox{ is }G\mbox{ independent}\}\] is a $\Pi^1_1$ total relation, so admits a $\Delta^1_1$ uniformization by \cite[Theorem 2.15]{edst}. The uniformizing function is a countable coloring.

Let $X=\baire\sm \bigcup\{A\in \Delta^1_1: A\mbox{ is }G\mbox{ independent}\}$. Note that $X$ is $\Sigma^1_1$. Suppose $X\not=\emptyset.$ (If we were to try to prove a dichotomy result here we would start trying to build some generic obstruction to colorability.) Suppose toward contradiction $f$ is a Borel countable coloring of $G$.

Recall that $\bb{P}_i$ is Gandy--Harrington forcing on $\baire^i$. By absoluteness, the interpretation of the code for $f$ remains a Borel coloring in any extension by a $\bb{P}_i$-generic. We can find some $A\in\bb P_1$ below $X$ so that $A\Vdash f(\dot x)=i$, where $\dot x$ is a name for the generic real. If $A$ is not independent, then ${B=\{(x,y): (x,y)\in G, x, y\in A\}}$ is a condition in $\bb P_2$ and \[B\Vdash f(\dot x)=f(\dot y)\; \wedge \; (\dot x,\dot y)\in G,\] which is a contradiction. Since being independent is $\Pi^1_1$ on $\Sigma^1_1$, $A$ is contained in some $\Delta^1_1$ independent set. But this means $A\cap X=\emptyset$, contradicting our choice of $A.$ %include proof that GH conditions force what they say ?
\end{proof}

\begin{cor}
The set of Borel countably colorable graphs is $\Pi^1_1$ in the codes.
\end{cor}
\begin{proof} Write $B_c$ for the set coded by a Borel code $c$, and write $\phi_n^x$ for the partial function computed by the $n^{th}$ oracle machine with oracle $x$.

For any $g\in \baire$, $g$ codes a graph with a countable coloring if and only if $g$ codes a graph $G$, and 
\[(\exists f\in \Delta^1_1(g))\; f\mbox{ is a countable coloring of }G.\] In detail, this second clause says
\[(\exists n\in \omega) \; \phi^g_n \mbox{ is a Borel code, and }B_{\phi_n^g}\mbox{ is a countable coloring of }G.\] 
And $f$ is a countable coloring of $G$ if and only if \[f\subseteq \baire\times \omega\wedge (\forall x,y,n,m) (y,n),(x,n), (x,m)\in f\rightarrow \left(x,y\not\in G\wedge n=m\right).\] This all $\Pi^1_1$.
\end{proof}

We will only state results for $\Delta^1_1$ sets below, but every statement relativizes and gives a complexity bound as above.

The general outline of the method is as follows: We first reduce whatever problem is at hand to the problem of covering some set $X$ by countably many Borel sets satisfying some property $\Phi$ (this is usually the combinatorial heart of the problem). Then we suppose that some $X\in\Delta^1_1$ can be covered by Borel sets $\ip{A_i:i\in\omega}$ satisfying $\Phi$, but not by $\Delta^1_1$ sets. We get a contradiction by forcing below \[\widetilde X := X\sm\bigcup\{A\in\Delta^1_1: \Phi(A)\}\] to find some $p$ and $i$ with \[p\Vdash \dot x\in A_i.\] We then show by contradiction that in fact $\Phi(p)$, so by a reflection argument $p\cap X=\emptyset.$ In the next section, we formalize this proof sketch and offer some technical variations on the idea.

Our second example involves a case where we want our covering sets to satisfy a notion of independence and a kind of closure. We can give a much simplified proof of the effectivization consequences of Harrington--Kechris--Louveau \cite[Theorem 5.2.7]{miller survey} by considering boxes in $\baire^2$ which avoid a given equivalence relation.

\begin{thm} \label{smooth}
If $E$ is a $\Delta^1_1$ equivalence relation and Borel reducible to $(=_{2^\omega})$, then $E$ is $\Delta^1_1$ reducible to $(=_{2^\omega})$.
\end{thm}

\begin{proof}
First note that $E$ is $\Delta^1_1$ (Borel) reducible to $(=_{2^\omega})$ if and only if $E^c$ is covered by a countable $\Delta^1_1$ (Borel) family of boxes which all avoid $E$. Indeed, if $f$ is a reduction, then we can set \[A_i=\{x: f(x)(i)=0\}\mbox{ and }B_i=\{x: f(x)(i)=1\}\] and consider the boxes $A_i\times B_i$. Conversely, if $E^c$ can be covered by $\Delta^1_1$ boxes, then 
\[(\forall (x,y)\not\in E) (\exists n\in\omega)\; n\in \s C, (x,y)\mbox{ is in the set coded by }n\] where $\s C=\{n\in \omega: n\mbox{ is a code for a $\Delta^1_1$ box avoiding }E\}$. So, by $\Pi^1_1$ on $\Pi^1_1$ reflection there is a $\Delta^1_1$ subset of $\s C$, say $\s D$ with the same property. We can enumerate $\s D$ to get a $\Delta^1_1$ covering sequence $\ip{A_i\times B_i:i\in\omega}$. By the second reflection theorem \cite[Lemma 1.4] {HMS}, we can assume each $A_i$ is $E$-invariant and avoids $B_i$, and we can define a reduction by $f(x)(i)=0$ if and only if $x\in A_i$.

Suppose $E\in \Delta^1_1$, $X=E^c\sm \bigcup \{A\times B: A,B\in \Delta^1_1, \left(A\times B\right)\cap E=\emptyset\}$ is nonempty, and $E^c=\bigcup_i A_i\times B_i$ with $A_i,B_i$ Borel. A set $Y\subseteq \baire^2$ is a box if and only if
\[(\forall (x,y),(a,b)\in Y)\;(x,b)\in Y.\] So again by the second reflection theorem applied to 
\[\Phi(Z,Y):=\left(\forall (x,y),(a,b)\not\in Z\right)\; (x,b)\in Y\mbox{ and }(x,y)\not\in E ,\] $X$ does not meet any $\Sigma^1_1$ boxes avoiding $E$. Also note that $X$ is $\Sigma^1_1.$

Since $X\Vdash (\dot x,\dot y)\not\in E$, there is some $p\subseteq X$ so that \[p\Vdash \dot x\in A_i\;\wedge \; \dot y\in B_i.\] And, since any $\bb P_2$ generic $(x,y)$ has $x$ and $y$ separately $\bb P_1$-generic, it must be that $A=\pi_1(p)\Vdash \dot x\in A_i$ and $B=\pi_2(p)\Vdash\dot x\in B_i$. So, $A\times B$ must avoid $E$ or else $\left(A\times B\right)\cap E\Vdash (x,y)\in E\cap \left(A_i\times B_i\right)$.

But then $p\subseteq A\times B\cap X=\emptyset$, which is a contradiction.
\end{proof}
\begin{cor}
The set of smooth relations is $\Delta^1_1$ in the codes.
\end{cor}
\begin{proof}
A real $c$ codes a smooth relation $E$ if and only if 
\[(\exists f\in \Delta^1_1)(\forall x,y) f(x)=f(y)\leftrightarrow x E y.\] As above, $(\exists f\in \Delta^1_1)$ is a equivalent to a universal quantifier over $\baire.$
\end{proof}

Our last example involves an ordinal length construction. It turns out we can also effectivize the ordinals length of the sequence to lie below $\omega_1^{CK}$. This gives another proof of the effectivization consequences of a dichotomy due independently to Kanovei and Louveau \cite[Theorem 5.2.3]{miller survey}.

\begin{dfn} 
For an ordinal $\alpha$, $\leq_{lex}^\alpha$ is the lexicographic order on $2^\alpha$, i.e.
\[x\leq_{lex}^\alpha y:\lra x=y \mbox{ or }(\exists \beta<\alpha)\left[ (x\res \beta)=(y\res \beta)\mbox{ and }x(\beta)<y(\beta)\right]\]
\end{dfn}

\begin{thm} \label{linearization}
Suppose $R$ is a quasi-order on $\baire$ with a Borel homomorphism to $\leq_{lex}^{\alpha}$ for some $\alpha<\omega_1$ which induces an injection on $\baire/\equiv_R$. Then there is a $\Delta^1_1$ homomorphism of $R$ to $\leq_{lex}^{\alpha'}$ for some $\alpha'<\omega_1^{CK}$ which induces an injection on $\baire/\equiv_R$.
\end{thm}
\begin{proof}

Given a family of functions $\s F$ on $\baire$, define an equivalence relation by \[x\equiv_{\s F}y \;:\lra\;(\forall f\in \s F)\; f(x)=f(y).\]  And for a function $f$, let $(\equiv_f)=(\equiv_{\{f\}})$. We want to find some $\Delta^1_1$ homomorphism from $R$ to a lexicographic ordering so that $(\equiv_R)=(\equiv_f)$.

Consider \[\s F=\{f: (\exists \alpha<\omega_1^{CK})\;f\mbox{ is a $\Delta^1_1$ homomorphism from $R$ to }\leq_{lex}^\alpha \}.\] Note that $\s F$ is $\Pi^1_1$ and $(\equiv_R)\subseteq (\equiv_{\s F})$. If $(\equiv_R)=(\equiv_\s F)$, then since this is a $\Pi^1_1$ on $\Pi^1_1$ statement about subsets of $\s F$, by reflection there is a $\Delta^1_1$ $\s G\subseteq \s F$ so that $(\equiv_{\s G})=(\equiv_R)$. We can enumerate $G$ as $\ip{f_i: i\in\omega}$. So each $f_i$ is a homomorphism from $R$ to some $\leq_{lex}^{\alpha_i}$ with $\alpha_i<\omega_1^{CK}$. Then $(\equiv_R)=(\equiv_f)$ where 
\[f(x)=f_0(x)\fr f_1(x)\fr f_2(x)\fr ... .\]

Now suppose towards contradiction that $X=(\equiv_{\s F})\sm (\equiv_R)$ is nonempty, and fix a Borel homomorphism $g$ from $R$ to some $\leq_{lex}^{\tilde \alpha}$. We have that 
\[X\Vdash g(\dot x)(\alpha)\not=g(\dot y)(\alpha) \mbox{ for some }\alpha.\] So we can define \[\alpha_0=\min\{\alpha: (\exists p\subseteq X)\; p\Vdash g(\dot x)(\alpha)\not=g(\dot y)(\alpha)\}. \] Note that $X\Vdash g(\dot x)\res \alpha_0=g(\dot y)\res \alpha_0$. 

Choose $p\subseteq X$ witnessing the above formula. Without loss of generality, we may assume $p\Vdash g(\dot x)(\alpha_0)=1\; \wedge\; g(\dot y)=0$. Let $A=\pi_1(p)$ and $B=\pi_2(p)$. Then, if $q=(A\times B)\cap R\cap (\equiv_{\s F})\not=\emptyset$,
\[q\Vdash g(\dot x)(\alpha_0)>g(\dot y)(\alpha_0)\;\wedge\; g(\dot x)\res \alpha_0=g(\dot y)\res \alpha_0\;\wedge \; x R y\] which contradicts the fact that $g$ is a homomorphism. 

So, $(A\times B)\cap R\cap (\equiv_{\s F})=\emptyset$. By the second reflection theorem we may assume $A$ is $\Delta^1_1$ and closed upwards under by $R\cap(\equiv_{\s F})$ and avoids $B_i$. By reflection there is a $\Delta^1_1$ sequence of functions $\ip{f_i: i\in\omega}$ so that each $f_i$ is in $\s F$ and $(A\times B)\cap R\cap (\equiv_{\{f_i:i\in\omega\}})=\emptyset.$ 

Suppose $f_i$ is a homomorphism into the lexicographic order on $\alpha_i$. We can define a homomorphism into the lexicographic order on $\gamma:=(\sum_i \alpha_i)+1$ as follows. Set $f_{\infty}(x)=1$ if and only if $x\in A$ and
\[f(x):=\left(f_0(x)\fr f_1(x)\fr ... \right)\fr f_{\infty}(x).\] To check this is a homomorphism, suppose $x R y$. If $f_i(x)<_{lex}^{\alpha_i}f_i(y)$ for some $i$, then $f(x)\leq_{lex}^\gamma f(y)$. Otherwise $x\equiv_{\s F}y$, and so if $x\in A$ then $y\in A$ by our closure assumption. This means $f_{\infty}(x)\leq f_{\infty}(y)$. In any case $f(x)\leq_{lex}^{\gamma}f(y)$ and $f$ is a homomorphism.

But then $f\in \s F$ and $f(x)\not=f(y)$ for $x,y\in p$, which contradicts our choice of $p$.
\end{proof}

\section{General Results}\label{general section}

Our first goal in this section is to prove a general result on effectivizing countable covers by Borel sets satisfying some property, $\Phi$. We can do this when $\Phi$ is conjunction of what we call independence properties and closure properties. These notions are defined below, but an informal description is as follows: $\Phi$ is an independence property if $\Phi(A)$ says all points in $A$ satisfy a combinatorial relation (e.g.~$A$ is $G$-independent for a graph $G$), and $\Psi$ is a closure property if $\Psi(A)$ says that $A$ contains all points which stand in a combinatorial relation with points from $A$ (e.g.~$A$ is a box).

\begin{dfn}

We say that a property $\Phi(A)$ of a set in $\baire$ is an \textbf{independence property} if there is some $\Delta^1_1$ property $\phi$ such that 
\[\Phi(A)\lra \neg (\exists \mathbf{x},y )\; \mathbf{x}\in A^k\mbox{ and }\phi(x_1,...,x_n,y).\]
\end{dfn}

Note that if $\Phi$ is an independence property then $\Phi$ is $\Pi^1_1$ on $\Sigma^1_1$. In practice, many $\Pi^1_1$ on $\Sigma^1_1$ properties are independence properties. But the two notions are not equivalent, and it is unclear how far beyond independence properties the results below generalize.

\begin{dfn}
 A property $\Psi(A)$ of sets is a \textbf{closure property} if there is a $\Delta^1_1$ property $\psi$ so that $\Psi(A)$ if and only if
 \[(\forall \mathbf{x}\in A^k)(\forall y,z) \; \psi(\mathbf{x},y,z)\rightarrow z\in A.\] For a closure property $\Psi$ and set $A$, the $\Psi$-closure of $A$ is $A^{\Psi}:=\bigcup_m f^m(A)$, where
 \[f(A):=A\cup \{z: (\exists \mathbf{x}\in A^k)(\exists y) \psi(x,y,z)\}.\]
\end{dfn}

Of course, the $\Psi$-closure of any set satisfies $\Psi$. And, if $A$ is $\Sigma^1_1$, then so is $A^{\Psi}$.

The key point about independence and closure properties (and conjunctions thereof) is that any Gandy--Harrington condition which forces a generic to be in set with one of these properties is contained in a $\Delta^1_1$ set with the same property. 

\begin{dfn}
 A property $\Phi(A)$ is \textbf{reflectable} if it is $\Pi^1_1$ on $\Delta^1_1$, $\Phi(B)$ is absolute between $V$ and $V[G]$ for any $\bb P_i$-generic $G$ and Borel set $B$, and the following condition holds: whenever $B$ is Borel, $A\in \bb P_1$, $A\Vdash \dot x\in B$, and $\Phi(B)$, there is some $\tilde A\in \Delta^1_1$ so that $A\subseteq \tilde A$ and $\Phi(\tilde A)$.
\end{dfn}

Most natural $\Pi^1_1$ on $\Delta^1_1$ properties (including closure and independence properties) will be $\bp11$ on $\bd11$ as well, so the absoluteness assumption in this definition is usually automatic.

\begin{lem}\label{theft} The following properties are all reflectable:
\begin{enumerate}
    \item $\bigvee_{i\in \omega} \Phi_{i\in\omega}(A)$, where $\ip{\Phi_i:i\in\omega}$ is a $\Delta^1_1$ sequence of reflectable properties
    \item $\Phi(f[A])$, where $\Phi$ is reflectable and $f:\baire\rightarrow\baire$ is a $\Delta^1_1$ bijection
    \item $\Phi\wedge \Psi$ where $\Phi$ is an independence property and $\Psi$ is a closure property.
\end{enumerate}
\end{lem}
\begin{proof} Checking $(1)$ and $(2)$ is routine. For $(3)$, suppose $A\Vdash \dot x\in B$ and $(\Phi\wedge \Psi)(B)$. We first show by induction on $m$ that $f^m(A)\Vdash \dot x\in B$, with $f$ as in the definition of $A^\Psi$. The base case $m=0$ is our assumption that \[f^0(A)=A\Vdash \dot x\in B.\] 

Suppose toward contradiction that $f^m(A)\Vdash \dot x\in B$, but that there is some $p\subseteq f^{m+1}(A)$ so that $p\Vdash \dot x\not\in B.$ We have nonempty condition \[q:=\{(\mathbf{x},y,z):\mathbf{x}\in \left(f^m(A)\right)^k,\; z\in p,\mbox{ and}\;\psi(\mathbf{x},y,z) \}\in \bb P_{k+2}.\] And, by the induction hypothesis \[q\Vdash \dot{ \mathbf{x}}\in B^k\;\wedge\;\dot z\not\in B\;\wedge \;\psi(\dot{\mathbf{x}},\dot y,\dot z)\] contradicting the fact that $\Psi(B)$.

Now suppose that $A\Vdash \dot x\in B$ and $(\Phi\wedge \Psi)(B)$. From the above, we may assume $A^{\Psi}=A$. Suppose toward contradiction that $\neg \Phi(A)$. Then \[C=\{(\mathbf{x},y): \mathbf{x}\in A^k\mbox{ and } \phi(\mathbf{x},y)\}\in \bb P_{n\times k} \] and since $A\Vdash \dot x\in B$, \[C\Vdash \dot {\mathbf{x}}\in B^k\;\wedge\; \phi(\dot {\mathbf{x}},\dot y ) \] which contradicts the fact that $\Phi(B).$ 

So, $(\Phi\wedge \Psi)(A)$, and by the second reflection theorem applied to \[\Theta(X,Y):=(\forall \mathbf{x}\in (\baire\sm X)^k, y,z) \;\left(\psi(\mathbf{x},y,z)\rightarrow z\in Y\right)\mbox{ and }\; \neg \phi(\mathbf{x},y)\] there is a $\Delta^1_1$ set $\tilde A\supseteq A$ so that $(\Phi\wedge \Psi)(\tilde A).$

\end{proof}

With the lemma above in hand, we can quickly prove a general effectivization result for countable covers by independent sets.

\begin{thm}\label{general}
Fix a reflectable property $\Phi$. Suppose $X\in \Delta^1_1$ and $X\subseteq\bigcup_{i\in\omega} B_i$, where each $B_i$ is a Borel set so that $\Phi(B_i)$. Then there is a $\Delta^1_1$ sequence of $\Delta^1_1$ sets $\ip{A_i: i\in\omega}$ so that $\Phi(A_i)$ for all $i$ and $X\subseteq\bigcup_{i\in \omega} A_i$.
\end{thm}

\begin{proof}

Suppose toward contradiction $X$ admits no $\Delta^1_1$ sequence as described. Consider the following set: \[X'=X\sm \bigcup\{A\in \Delta^1_1 : \Phi(A)\}.\] Note that $X'$ is $\Sigma^1_1$.

By assumption, $X'$ is nonempty, so it is in $\bb P_1$. We have
\[X'\Vdash (\exists i) \; \dot x\in B_i\] so for some $i\in\omega$ and $A\subseteq X',$ \[A\Vdash \dot x\in B_i.\] But then, $\Phi(\tilde A)$ for some $\Delta^1_1$ set $\tilde A\supseteq A$. So $A$ avoids $X$, which contradicts our choice of $A$. 
\end{proof}

\subsection{Minor Technical Variations}                                          

We want two minor technical variations on this theorem. The first of these variations says that we can effectivize countable Borel families whose projections cover a given space, provided the projection map has countable fibers.

\begin{thm} \label{technical variant}
Suppose that $\Phi(A)$ is a reflectable property, $R\subseteq \baire^{2}$ is $\Delta^1_1$ with countable sections, and $\Phi(A)$ implies $A\subseteq R$. If $X\subseteq \baire^n$ admits a countable cover by Borel sets of the form $\pi_1(A)$ with $\Phi(A)$, then $X$ admits a cover by $\Delta^1_1$ sets of this form.
\end{thm}
\begin{proof}
By the effective Feldman-Moore theorem (see Theorem \ref{effectivelnfm} and the preceding comments in Section \ref{applications} below), there is a $\Delta^1_1$ sequence of involutions $\ip{f_i:i\in\omega}$ so that, for any $x\in \baire$ and $(x,y)\in R$, \[\{(x,y'): (x,y')\in R\}=\{f_i(x,y): i\in\omega\}.\] Then, $X=\bigcup_i \pi_1(A_i)$ if and only if \[\{(x,y)\in R: x\in X\}=\bigcup_i\bigcup_j \{f_j(x,y): (x,y)\in A_i\}.\]

So, $X$ admits a cover by $\Delta^1_1$ sets of the form $\pi_1(A)$ with $\Phi(A)$ if and only if $\widetilde X=\{(x,y)\in R: x\in X\}$ admits a cover by $\Delta^1_1$ sets so that $(\exists i)\Phi(f_i[A])$. By Lemma \ref{theft}, $(\exists i)\Phi(f_i[A])$ is reflectable, and we can apply the previous theorem.
\end{proof}

The second variation generalizes beyond covering spaces with unions of sets to covering spaces with unions of intersections of unions of sets. One could extend this further to arbitrary finite alternations of unions and intersections with mostly notational changes.

\begin{thm}\label{silly generalization}
Fix a $\Delta^1_1$ sequence of reflectable properties $\ip{\Phi_{i,j}:i,j\in\omega}$ and for $i,j\in \omega$. Suppose there is a family of Borel sets $\ip{B_{i,j,k}: i,j,k\in\omega}$ such that \[\widetilde X=\bigcup_{i} \bigcap_{j}\bigcup_{k} B_{i,j,k},\] where $\widetilde X\subseteq \baire$ is $\Delta^1_1$ and for all $i,j,k\in \omega$ \[\Phi_{i,j}(B_{i,j,k}).\] Then there is a $\Delta^1_1$ such family.
\end{thm}
\begin{proof}
Let \[X=(\widetilde X\times \omega^2)\sm \bigcup\{ A\times \{(i,j)\} : A\in \Delta^1_1, (i,j)\in \omega^2,\Phi_{i,j}(A)\}.\] If we have that
\[(\forall x )(\exists i )(\forall j )\; (x,i,j)\not\in X,\] then we can find our $\Delta^1_1$ family. So suppose towards contradiction that
\[(\exists x)(\forall i )(\exists j )\; (x,i,j)\in X.\]  We can find a condition $p\subseteq \{x: (\forall i)( \exists j )\;(x,i,j)\in X\}$, and an index $\tilde i$ so that \[p\Vdash (\forall j)(\exists k)\; \dot x\in B_{\tilde i, j,k}.\] Refining $p$, we can find some $\tilde j$ so that 
\[p\subseteq \{x: (x,\tilde i,\tilde j)\in X\}.\] But then, refining $p$ again, we can find $\tilde k$ so that 
\[p\Vdash \dot x\in B_{\tilde i,\tilde j,\tilde k}.\] But then $p$ is contained in a $\Delta^1_1$ set $\tilde p$ so that $\Phi(\tilde p)$, and $\tilde p\times\{(\tilde i,\tilde j)\}\cap X=\emptyset$, which contradicts the fact that $p\subseteq \{x: (x,\tilde i,\tilde j)\in X\}$.
\end{proof}

\subsection{Transfinite constructions}

Our last variation is somewhat more complicated and covers problem where we want to cover an object with some ordinal length sequence of sets. This generalizes Theorem \ref{linearization}.

\begin{dfn}
Given $\overline A=\ip{A_{\alpha}:\alpha<\beta}$ a countable ordinal length sequence of sets in $\baire$ and $X\subseteq \baire$, define 
\[X_{\overline A}=X\sm \bigcup_{\alpha<\beta} A_{\alpha}.\] And, given $\s F$ a set of such sequences, let \[X_{\s F}=X\sm \bigcup_{\overline A\in \s F}\bigcup_{\alpha<\beta} A_{\alpha}=\bigcap_{\overline A\in \s F} X_{\overline A}.\]

Abusing notation slightly, identify the ordinal product $\omega\cdot \alpha$ with the lexicographic order on $\alpha\times \omega$. For $\overline A=\ip{A_{\beta}:\beta\in \omega\cdot \alpha}=\ip{A_{\beta, i}: (\beta,i)\in \alpha\times \omega}$, set $\overline A_0=\ip{A_{\beta, 0}: \beta\in \alpha}.$

Say that a property $\Phi(\ip{ A_{\beta}: \beta\in\omega\cdot \alpha})$ of an ordinal length sequence of sets (with length of the form $\beta\times \omega$) is a \textbf{refinement} property if for some $\Delta^1_1$ sequence of properties $\ip{\phi_i:i\in\omega}$ and $\Delta^1_1$ sequence of closure properties $\ip{\Psi_i:i\in\omega}$, $\Phi(\overline A)$ if and only if
    \[(\forall \beta, i)\left[ \Psi_i(A_{\beta, i})\wedge \left(\forall \mathbf{x}\in \baire_{(\overline A_0\res \beta)}^k\forall y\right)\; \mathbf{x}\in A_{\beta,i}^k\rightarrow\phi_i(\mathbf{x},y) \right] \]
\end{dfn}

Informally, a refinement property says an independence property holds for each $A_{\alpha,i}$ if we ignore points in earlier $A_{\alpha,0}$'s, and each $A_{\alpha,i}$ is $\Psi_i$-closed. For any refinement property property $\Phi$, the set ${\{\overline A\in\Delta^1_1: (\forall\beta )\; A_\beta\in \Delta^1_1\mbox{ and }\Phi(\overline A)\}}$ is $\Pi^1_1.$ 

If some countable collection of sequences satisfies a refinement property, so does the sequence obtained by concatenating them together. We will often use this in conjunction with the following proposition.

\begin{prop} \label{splice} Suppose $\s F$ is a $\Pi^1_1$ family of $\Delta^1_1$ sequences which is closed under $\Delta^1_1$ concatenations. For any $\Sigma^1_1$ set $p\subseteq \baire^n$, whenever $p\cap X_{\s F}^n=\emptyset$ there is a single $\Delta^1_1$ sequence $\overline A\in\s F$ so that $p\cap X_{\overline A}^n=0$.
\end{prop}
\begin{proof}
For any $\Sigma^1_1$ sets $X$ and $p$, $p\cap (X_{\s F})^n=\emptyset$ if and only
\[ (\forall \mathbf{x})\;\mathbf{x}\not\in p\mbox{ or }(\exists \overline A_1,...,\overline A_n\in \Delta^1_1)(\forall i)\left[ \overline A_i\in \s F\mbox{ and }(\exists \alpha_i)\; x_i\in (A_i)_{\alpha_i}\right].\] As a property of $\s F$, this is $\Pi^1_1$ on $\Pi^1_1$. So, if $\s F$ is a $\Pi^1_1$ collection of sequences with $p\cap X_{\s F}^n=\emptyset$, we can find a $\Delta^1_1$ subset $\s G\subseteq \s F$ so that $p\cap X_{\s G}^n=\emptyset$. We can find a $\Delta^1_1$ enumeration of the sequences in $\s G$ and concatenate them together to get a single sequence $\overline A$ so that $p\cap X^n_{\overline A}=\emptyset.$
\end{proof}

We can prove that if there is some family of ordinal length sequence of Borel sets satisfying a refinement property that covers a $\Delta^1_1$ space, then there is a $\Delta^1_1$ such sequence (whose length is an effective ordinal).

\begin{thm} \label{transfinite}
Suppose that $\Phi$ is a refinement property, $X\subseteq \baire^n$ is $\Delta^1_1$, and there a sequence of Borel sets $\overline B=\ip{B_{\beta}: \beta\in \omega\cdot \alpha}$ with $X_{\overline B}=\emptyset$, $\alpha<\omega_1$, and $\Phi(\overline B)$. Then, there is a $\Delta^1_1$ sequence of $\Delta^1_1$ sets $\overline A=\ip{ A_{\beta}: \beta\in\omega\cdot \alpha'}$ so that $X_{\overline A}=\emptyset$, $\Phi(\overline A)$ and $\alpha'<\omega_1^{CK}$.
\end{thm}

\begin{proof}
First, define $\s F$ by \[\s F=\{\overline A\in \Delta^1_1:\Phi(\overline A, )\mbox{ and } (\forall \beta)\; A_\beta\in \Delta^1_1\}.\] If $X_{\s F}=\emptyset$, then by Proposition \ref{splice}, there is single $\Delta^1_1$ sequence $\overline A$ so that $X_{\overline A}=\emptyset$. So, suppose toward contradiction that $X_{\s F}\not=\emptyset$. Note that $X_{\s F}$ is $\Sigma^1_1.$ So, 
\[X_{\s F}\Vdash \dot x\in B_{\alpha,i}\mbox{ for some }\alpha.\] Define \[\alpha_0:=\min\{\alpha: (\exists p\subseteq X_{\s F},i\in\omega)\; p\Vdash \dot x\in B_{\alpha,i}\}\] and fix some $i\in\omega$ and $p\subseteq X_{\s F}$ so that $p\Vdash \dot x\in B_{\alpha_0,i}$. By the proof of Lemma \ref{theft}, we may assume $\Psi_i(p)$.

If $q=\{(x_1,...,x_k)\in p^k:\neg \phi_i(x)\}\cap X_{\s F}$ is nonempty, then \[q\Vdash \dot x_1,...,\dot x_k\in B_{ \alpha_0,i}\wedge \;\neg\phi_i(x).\] Since $\Phi(\overline B)$, by absoluteness $q\Vdash (\exists j,\beta<\alpha_0)\; x_j\in B_{ \beta,0}$. But then we can find a condition below $p$ forcing $\dot x\in B_{\beta,0}$ with $\beta<\alpha_0$, contradicting minimality.

So, $\{(x_1,...,x_k)\in p^k: \phi_i(x)\}\cap X_{\s F}^k=\emptyset$, and by Proposition \ref{splice} there is a single $\Delta^1_1$ sequence $\overline A$ so that $\{(x_1,...,x_k)\in p^k: \phi(x)\}\cap X_{\overline A}^k=\emptyset$. But then by reflection we can find a $\Delta^1_1$ set $A'\supseteq p$ so that $\Psi_i(A')$ and $\phi_i(\mathbf x)$ for $\mathbf{x}\in (A')^k\cap (\baire_{\overline{A}_0})^k$. Then, setting $A'_j=A'$ if $j=i$ and $A'_j=A_{0,j}$ otherwise, we get a $\Delta^1_1$ sequence of sets $\overline A'=\ip{A'_j:j\in\omega}$ so that $p\subseteq A_i$ and $\overline A\fr \overline A'\in \s F$, contradicting our choice of $p$.
\end{proof}

The only application of this theorem we give below is to local colorings of graphs equipped with quasi-orders. However, we point out that the ``forceless, ineffective, and powerless" proofs of Miller typically work by marrying sequences of Borel sets satisfying a refinement property to the usual derivative process on a tree of attempts to build a homomorphism from a canonical obstruction (so that there is no homomorphism if and only if the tree is well-founded if and only if there a sequence of refining sets that cover $\baire$). So Theorem \ref{transfinite} should usually apply whenever Miller's methods apply.

\section{Applications} \label{applications}

In this section we give several applications of the results above to the study of Borel graphs and relations. Throughout, we assume that our graphs and relations are on $\baire$. The complexity results lift to other spaces by standard universality arguments. We start by giving poofs of two folklore theorems.

The classical Lusin--Novikov uniformization theorem says that any Borel relation with countable sections is the union of countably many Borel partial functions \cite[Theorem 18.10]{cdst}. One important corollary is the Feldman--Moore theorem, which says that any countable Borel equivalence relation is induced by a Borel group action and in fact is a union of countably many Borel involutions \cite[Theorem 1.3]{km}. There are elegant effective proofs of these theorems based on the effective perfect set theorem, but these effective proofs do not seem to be recorded in the literature. Here we show how the effective versions of these theorems can be deduced from the classical versions. 

\begin{thm} \label{effectivelnfm}
If $R$ is $\Delta^1_1$ relation with countable sections, then $R=\bigcup_i f_i$ where $\ip{f_i:i\in\omega}$ is a $\Delta^1_1$ partial function. Further, if $R$ is an equivalence relation, then each $f_i$ can be chosen to be an involution.
\end{thm}
\begin{proof}
The classical Lusin--Novikov theorem tells us that $R$ is a union of Borel partial functions. A set $f\subseteq\baire^2$ is a partial function if and only if
\[(\forall x,y,y')\;(x,y),(x,y')\in f\rightarrow y=y'.\] This is an independence property, so by Theorem \ref{general} $R$ is a union of a $\Delta^1_1$ sequence of $\Delta^1_1$ partial functions.

And if $R$ is an equivalence relation, then the classical Feldman--Moore theorem says that $R$ is a union of Borel involutions. A set $f\subseteq\baire^2$ is a partial involution if and only if
\[f\mbox{ is a function}\mbox{ and }(\forall x,y)\;(x,y)\in f\rightarrow (y,x)\in f.\] This is a conjunction of an independence property and a closure property. So again by Theorem \ref{general}, $R$ is a union of a $\Delta^1_1$ sequence of $\Delta^1_1$ partial involutions. 

To finish, we need to show that any partial $\Delta^1_1$ involution extends to an involution. The domain of any $\Delta^1_1$ partial function is $\Delta^1_1$, so we can then extend a partial involution $f$ by setting $f(x)=x$ for $x\not\in \dom(f)$.
\end{proof}

Next, we can characterize graphs which can be (not necessarily symmetrically) generated by countable families of functions. That is, we reprove \cite[Theorem 4.1]{orientations}. 

\begin{dfn}
A graph $G$ is \textbf{generated} by a list of functions $\ip{f_i: i\in I}$ if $G\subseteq\{(x,y): y=f_i(x)\mbox{ or }x=f_i(y)\}.$
\end{dfn}

We could equivalently require $G=\{(x,y): y=f_i(x)\mbox{ or }x=f_i(y)\}$ if we allow partial functions. Since any $\Delta^1_1$ partial function extends to a $\Delta^1_1$ total function, this equivalence holds even when we restrict to Borel or $\Delta^1_1$ functions.

\begin{thm} \label{manyfunctions}
If $G$ is a $\Delta^1_1$ graph generated by a countable list of Borel functions, then $G$ is generated by a countable list of $\Delta^1_1$ functions.
\end{thm}
\begin{proof}
A graph $G$ is generated by a countable list of functions if and only if $G$ can be covered by countably many sets $f$ so that $f$ is a partial function or $-f$ is a partial function (where $-(x,y)=(y,x)$). Since being a function is an independence property and $e\mapsto -e$ is a $\Delta^1_1$ bijection, Lemma \ref{theft} tells us that this is a reflectable property.
\end{proof}
\begin{cor}
The set of Borel graphs generated by a countable family of Borel functions is $\Pi^1_1$ in the codes.
\end{cor}

We can prove a similar result for graphs generated by a single function. This argument is inspired by Hjorth and Miller's dichotomy for end selection, but we need to make use of some features of acyclic graphs to get around their assumption that $G$ is locally countable. This settles \cite[Problem 4.5]{orientations}.

\begin{thm} \label{1function}
If $G$ is $\Delta^1_1$ and generated by a single Borel function, then it is generated by a single $\Delta^1_1$ function.
\end{thm} 
\begin{proof}
For this theorem, the reduction to a covering problem is nontrivial. Any graph generated by a single function is a sidewalk.\footnote{A sidewalk is a graph with at most one cycle in each component. These are sometimes called pseudoforests.} Let \[A=\{x\in \baire: [x]_G\mbox{ contains a cycle}\}.\] Each component contains at most one cycle, so this set is $\Delta^1_1$, and $G\res A$ admits a $\Delta^1_1$ selector. From this, it is straightforward to construct a function which generates $G\res A$. 
Therefore, without loss of generality we may assume that $G$ is acyclic. This means simple paths between points are unique, so the connectedness relation is $\Delta^1_1$.

\begin{dfn}
A (directed) edge $e=(e_0,e_1)$ \textbf{points} to a vertex $x$ if $x$ and $e_0$ are connected and the simple path from $e_0$ to $x$ goes through $e_1$. We write $ePx$

Write $eIe'$ if $e\not=e'$, $e_0,e'_0$ are connected, and neither $eP e'_0$ nor $e'P e_0.$ Write $e\prec e' $ if $e P e'_0$ but not $e'P e_0.$

A set of edges $A$ is \textbf{consistent} if $(\forall e,e'\in A) \;\neg( e I e')$.
\end{dfn}
The simple path from a vertex $x$ to itself contains only $x$. So, if $A$ is consistent, then we cannot have $(x,y)$ and $(x,y')$ both in $A$ unless $y=y'$. That is, any consistent set must be a partial function.
For any $\Delta^1_1$ consistent set $A$, we can take its downward closure under $\prec$ to get a $\Sigma^1_1$ consistent $\prec$-closed set. By the second reflection theorem, any $\Sigma^1_1$ consistent $\prec$-closed set extends to a $\Delta^1_1$ consistent $\prec$-closed set.

\begin{lem}
$G=\bigcup_i ( A_i\cup (- A_i))$ where $A_i$ is a $\Delta^1_1$ sequence of consistent sets if and only if $G$ is generated by a $\Delta^1_1$ function.
\end{lem}
\begin{proof}[Proof of lemma]
Suppose $f$ generates. We want to show that $f\cap G$ is consistent. Otherwise, there are some $e,e'\in f$ so that $eIe'$. If there are $n$ undirected edges on the simple path between $e_0$ and $e'_0$, there are $n-1$ vertices on the path besides $e_0$ and $e'_0.$ Label these vertices $v_1,...,v_{n-1}$. Since $f(e_0)$ and $f(e'_0)$ point off of the path and $f$ generates $G$, there can only be the $n-1$ edges on the path of the form $\{v_i, f(v_i)\}$. This is a contradiction. So if $f$ generates, then $G\subseteq f\cup (-f)$ gives a cover.

 For the other direction we may assume each $A_i$ is $\prec$-closed. Note that if $X,Y$ are $\prec$-closed and consistent, then so is $Y\cup (X\sm (- Y))$. Define $f_i$ inductively by $f_{i+1}=f_{i}\cup (A_{i+1}\sm (-f_i) ),$ and set $f=\bigcup_i f_i$. Then $f$ is a partial function so that $G=f\cup (-f)$, and $G$ is generated by a single function.
\end{proof}

So, $G$ is generated by a (Borel) $\Delta^1_1$ function if and only if $G$ is covered by (Borel) $\Delta^1_1$ sets $A$ such that $A$ is consistent and $-A$ is consistent.
\end{proof}
\begin{cor}
The set of graphs generated by a Borel function is $\Pi^1_1$ in the codes.
\end{cor}

We can also prove an effectivization result for end selection in locally countable graphs using an argument of Hjorth and Miller \cite{HM}. Note, their dichotomy was proven without effective methods and only gives a complexity upper bound of $\mathbf{\Delta}^1_2$ for end selection.

\begin{thm} \label{selection}
If $G$ is a locally countable $\Delta^1_1$ graph which admits a Borel end selection, then $G$ admits a $\Delta^1_1$ end selection
\end{thm}
\begin{proof}
Since $G$ is locally countable, $E_G$ is $\Delta^1_1$. Let \[R=\{(S,x): S\in [G]^{<\infty}, x\in [S]_G\}.\] Note that $R$ has countable sections over $[G]^{<\infty}$. Following \cite{HM}, $\s D\subseteq [G]^{<\infty}\times \baire$ is directed if
\begin{enumerate}
    \item $\s D\subseteq R$
    \item For all $(S,x)\in \s D$, $x\not\in S$
    \item If $(S,x), (S', x')\in \s D$ and $S$ and $S'$ are in the same $G$-component, then and $x',x$ are connected in at least one of $G\sm S$ or $G\sm S'$. 
\end{enumerate} %double check

Note that directedness is an independence property. And, \cite[Proposition 3.4]{HM} says a locally countable Borel graph $G$ admits a Borel end selection if and only $[G]^{<\infty}$ can be covered by projections of Borel directed sets. It is straightforward to check that one can replace Borel with $\Delta^1_1$ in their proof. We can then apply Theorem \ref{technical variant}.
\end{proof}

The theorem above implies that any $\Delta^1_1$ Schreier graph of a free Borel action of $\Z$ is the Schreier graph of a free $\Delta^1_1$ action. Or, in the language of \cite{directable forests}, the set of directable forests of lines is $\Pi^1_1$ in the codes. We can generalize this to effectivize Borel actions of $\Z^2$.

\begin{thm}\label{actions} If $G$ is a $\Delta^1_1$ graph and the Schreier graph of free Borel action of $\Z^2$ with the usual generating set, then $G$ is the Schreier graph of a free $\Delta^1_1$ action of $\Z^2.$
\end{thm}
\begin{proof} All components of $G$ isomorphic the square lattice. In particular $G$ is locally countable.

\begin{dfn}
A \textbf{straight line} is a directed path so that no other path of the same or shorter length has the same endpoints. A \textbf{rectangle} is a simple directed cycle that can be divided into 4 straight lines.

Two directed edges $e,e'$ are \textbf{parallel} if they lie on some straight line together or if $(- e)$ and $e'$ are on opposite sides of a rectangle. We write $e \| e'$. We say $e,e'$ are \textbf{anti-parallel} if $(- e) \| e$. And we say $e,e'$ are \textbf{perpendicular} if they are in the same component and are neither parallel nor anti parallel. We write $e\perp e'.$ \end{dfn}

Note that all of these relations are $\Delta^1_1.$

\begin{dfn} Let $\{a,b,-a,-b\}$ be the usual generating set for $\Z^2$. For a graph $G$, say that a partial function $f:G \rightarrow\{a, b, -a, -b\}$ is a partial diagram if for all $e, e'$, 
\begin{enumerate}
    \item if $e\|e'$, then $f(e)=f(e')$
    \item if $e\|(- e')$ then $f(e)=-f(e')$
    \item if $e\perp e'$ then $f(e)$ and $f(e')$ have different letters (i.e.~$f(e)\not=\pm f(e')$).
\end{enumerate}

A Cayley diagram is a partial diagram whose domain is all of $G$.
\end{dfn}

 We first check that being induced by a free $\Z^2$ action is equivalent to having Cayley diagram. If $a:\Z^2\curvearrowright\baire$ is an action generating $G$, then any straight line is of the form $\{n\gamma \cdot x: n\in\Z\}$ for some generator $\gamma$ of $\Z^2$. It follows that  $f(x,\gamma\cdot x)=\gamma$ defines a Cayley diagram. Conversely, given a Cayley diagram $f$, we can define an action of a generator $c$ by $c\cdot x=y\lra f(x,y)=c$. To see this extends to an action of $\Z^2$, note that for any $x,y$ $f(x,y)=-f(y,x)$ so $(\gamma -\gamma)\cdot x=x$, and $(x, a\cdot x, (b+a)\cdot x,(-a+b+a)\cdot x,(-b-a+b+a)\cdot x)$ must be a rectangle, so $x=(-b-a+b+a)\cdot x.$

Since being a partial diagram is an independence property, it suffices to show that any sequence of partial diagrams whose domains cover $G$ can be patched together into a Cayley diagram.

Suppose we have two $\Delta^1_1$ partial diagrams $f$ and $g$. We show how to modify and glue them together to get a partial diagram whose domain is the union of their domains. The result then follows by induction. Define $h$ by
\begin{enumerate}
\item  if $e\in \dom(f)$, then $h(e)=f(e)$
\item if $e\in \dom(g)\sm [\dom(f)]_G$, $h(e)=g(e)$
\item else if $e\in \dom(g)$ and $e\|e'$ for some $e'\in f$ (or $e\|(- e')$), then $h(e)=f(e')$ (or $-f(e')$).
\item else if $e\perp e'\in \dom(f)$, $h(e)$ has the sign of $g(e)$ and the opposite letter of $f(e')$.
\end{enumerate} It is straightforward to check that $h$ is a partial diagram. 
\end{proof}
\begin{cor}
The set of Schreier graphs of Borel $\Z^2$ actions is $\Pi^1_1$ in the codes.
\end{cor}

 A similar idea works for $\Z^n$. We can also effectivize a recent result of Miller which we will use in the next section. 

\begin{thm}[c.f. {\cite[Theorem 1]{millers new thing}}] \label{arrays}
If $\ip{G_{i,j}: i,j\in\omega}$ is a $\Delta^1_1$ array of graphs, and there are Borel sets $B_i$ such that $\baire=\bigcup_i B_i$ and $G_{i,j}\res B_i$ has a Borel countable coloring for every $j$, then there are $\Delta^1_1$ sets $A_i$ so that $G_{i,j}\res B_i$ has a $\Delta^1_1$ countable coloring for all $j$.
\end{thm}
\begin{proof}
This is equivalent to effectivizing sequences $\ip{A_{ijk}: i,j,k\in\omega}$ of sets so that $A_{ijk}$ is $G_{ij}$-independent and so that for all $x$ there is an $i$ so that for all $j$ there is a $k$ so that $x\in A_{ijk}$. The result then follows from Theorem \ref{silly generalization}.
\end{proof}
\begin{cor}
The set of arrays of Borel graphs $G_{i,j}$ so that there are Borel sets $B_i$ with $\chi_B(G_{i,j}\res B_i)\leq \aleph_0$ for all $i,j$ and $\baire=\bigcup_i B_i$ is $\Pi^1_1$ in the codes.
\end{cor}

Finally, we can effectivize Miller's dichotomy for local colorings of graphs equipped with quasi-orders \cite[Theorem 5.1.2]{miller survey}.

\begin{thm} \label{goho}
If $(G,R)$ is a $\Delta^1_1$ graph and quasi-order on $\baire$ so that $R$ admits a Borel homomorphism, $f$, to some $\leq_{lex}^\alpha$ with $\alpha<\omega_1$ with $G\cap (\equiv_f)$ countably Borel 
colorable, then $(G,R)$ admits a $\Delta^1_1$ such homomorphism into some $\leq_{lex}^\alpha$ with $\alpha<\omega_1^{CK}$.
\end{thm}
\begin{proof}
First we show that such a homomorphism from $R$ is equivalent to a sequence of pairs of sets $\ip{(R_\alpha, S_\alpha):\alpha<\beta}$ in $\baire^2$ with
\begin{enumerate}
    \item Each $R_\alpha$ is a box $A_\alpha\times B_\alpha$, and each $S_\alpha$ is a box $C_\alpha\times D_\alpha$ 
    \item $\left(R_\alpha\cap R\right)\sm \left(\bigcup_{\gamma<\alpha}R_\gamma\right)=\emptyset$
    \item $\left(C_\alpha^2\cap G\right)\sm \left(\bigcup_{\gamma<\alpha}R_\gamma\right)=\emptyset $
    \item $C_\alpha\cap D_\alpha=\emptyset$
    \item $G\subseteq \bigcup_{\alpha<\beta} R_\alpha\cup S_\alpha$.
\end{enumerate} 

Given such a sequence, we may assume each $A_\alpha$ is closed upwards under the quasi-order generated by $R\sm \left(\bigcup_{\gamma<\alpha}R_\gamma\right)$. So, $f(x)(\alpha)=\chi_{A_\alpha}(x)$ defines a homomorphism. And $c(x)=\min\{\alpha: x\in C_\alpha\}$ gives a coloring of $G\cap (\equiv_f)$. Conversely, given a homomorphism and coloring we can get such a sequence by considering the color classes and the sets of the form $\{x: f(x)(\alpha)=i\}$ for $i\in\{0,1\}$.

Since $(1),(2),(3),$ and $(4)$ describe a refinement property, we can apply Theorem \ref{transfinite}.
\end{proof}
\begin{cor}
The set of pairs $(G,R)$ where $G$ is a Borel graph and $R$ is a Borel quasi-order with a homomorphism and coloring as above is $\Pi^1_1$ in the codes.
\end{cor}

\section{Some dichotomies}

We end by establishing some dichotomy theorems related to some of the new effectivization results of the previous section. Our two results about generating graphs with functions can be deduced from a dichotomy for graphs equipped with group actions.

\begin{dfn}
For $G$ a Borel graph on $X$, $\Gamma$ a countable group, and $a$ an action of $\Gamma$ on $X$, we say that $f:\baire\rightarrow \omega\times \Gamma$ is an \textbf{a-coloring} of $(G,a)$ if, for all $n\in\omega$, $\{h\cdot_a x: f(x)=(n,h)\}$ is $G$-independent. 

Fix $\ip{(\gamma_n,s_n):n\in\omega}$ a sequence with $(\gamma_n,s_n)\in \Gamma\times 2^n$ such that, for each $\gamma\in \Gamma$, $\gamma=\gamma_n$ for infinitely many $n$, and every string in $2^{<\omega}$ is extended by some $s_n$. Call such a sequence \textbf{generic}. Then $G_\Gamma$ is the graph on $\Gamma\times 2^\omega$ defined by 
\begin{align*} x G_\Gamma y:\lra & (\exists z\in 2^\omega,n\in\omega) \left[x=s_n\fr 0 \fr z\mbox{ and }y=s_n\fr 1 \fr z \right].\end{align*}

For any $x\in 2^\omega,$ $a_\Gamma$ is the obvious action of $\Gamma$ on $\Gamma\times 2^\omega$:
\[\gamma\cdot (\gamma',z)=(\gamma\gamma',z).\]
\end{dfn}

Note that a graph $G$ is generated by a countably family of functions if and only if there is an a-coloring of $(G', a)$ where $G'$ is the graph on $G$ given by \[G'= \{(e,e')\in G^2: e_0=e'_0\}\] and $a$ is the action of $C_2=(\{\pm 1\},\times)$ given by
\[-1\cdot (x,y)=-(x,y)=(y,x).\] This along with the following theorem settles \cite[Problem 4.4]{orientations}. Similarly, a graph is generated by a single function if and only if there is an a-coloring of $(G'',a)$, where $eG'' e'$ if and only if $e C e'$ and $a$ is as before.

\begin{thm} \label{equivariant go}
For any Borel graph $G$ and Borel action of $a:\Gamma\curvearrowright\baire$, either $(G,a)$ admits a Borel a-coloring or there is an equivariant homomorphism from $G_\Gamma$ to $G$, i.e.~a map $f: \Gamma\times 2^\omega\rightarrow \baire$ so that 
\begin{enumerate}
    \item $f(h\cdot_{a_\Gamma} x)=h \cdot_a f(x)$
    \item If $x G_\Gamma y$ then $f(x) G f(y).$
\end{enumerate}
\end{thm}

\begin{proof}
First we show that the two options are mutually exclusive. If $c$ is an a-coloring of $(G,a)$ and $f$ is an equivariant homomorphism of $(G_\Gamma, a_\Gamma)$ into $G$, then $c\circ f$ is an a-coloring of $G_\Gamma.$ So, it suffices to show no such a-coloring exists. Suppose $c:\Gamma\times 2^\omega\rightarrow \omega\times \Gamma$ is a Borel a-coloring. Then, for some $(n,\gamma)$, $c\inv(n,\gamma)$ is nonmeager. Since $\Gamma$ acts by homeomorphisms, ${U=\gamma\cdot c\inv(n,\gamma)}$ is also nonmeager. Suppose $U$ is comeager in $N_{(\gamma', s)}$. By genericity of $\ip{(\gamma_n,s_n):n\in\omega}$, there is some $n$ with $\gamma'=\gamma_n$ and $s\subseteq s_n$. Let $t$ be the homeomorphism that flips the $(n+1)^{th}$ bit of a string. Then, every $(\gamma',\sigma)\in N_{(\gamma_n,s_n)}$ has a neighbour $(\gamma',t(\sigma))\in N_{\gamma',s}$. So, by Baire category, there is a nonmeager set of points in $U$ with a neighbour in $U$. This contradicts the definition of a-colorings.

For $\gamma\in \Gamma$, define $G_\gamma$ by $x\; G_\gamma\; y$ if and only if $\gamma\cdot x \;G \;\gamma\cdot y$. Then, there is an a-coloring of $(G,a)$ if and only if there a cover of $X$ by sets which are independent for some $G_\gamma$.

Let $G_{0,\gamma}$ be the graph on $2^\omega$ given by \[x \; G_{0,\gamma}\; y:\lra (\exists z\in 2^\omega, n\in\omega)\; \gamma_n=\gamma, x=s_n\fr 0\fr z, y=s_n\fr 1\fr z.\]

Suppose $(G,a)$ does not admit an a-coloring. By  \cite{millers new thing} and the genericity of $\ip{(\gamma_n,s_n):n\in\omega}$, there is some $g: 2^\omega\rightarrow G$ which is a homomorphism from $G_{0,\gamma}$ to $G_\gamma$  for every $\Gamma$.

Now define $f(\gamma,x)=\gamma\cdot g(x)$. We show this is an equivariant homomorphism from $(G_\Gamma, a_\Gamma)$ to $(G,a)$. Equivariance is clear from the definition, and if $(\gamma, x) G_\Gamma (\gamma, y)$, then $x \; G_{0,\gamma}\; y$, so $g(x)\; G_\gamma \; g(y) $ and $\gamma\cdot g(x)\; G \;\gamma\cdot g(y)$.
\end{proof}
\begin{cor}
The set of pairs of Borel graphs and actions which admit a-colorings is $\Pi^1_1$ in the codes.
\end{cor}
\begin{proof}
The proof above gives a reduction to the set of codes for sequences of graphs as in Theorem \ref{arrays}.
\end{proof}

So, a graph $G$ is generated by a countable family of Borel functions if and only if there is an equivariant homomorphism from $(G_{C_2}, a_{C_2})$ to $(\widetilde G, a)$ defined above. We can refine this analysis somewhat to eliminate reference to the auxiliary graph $\widetilde G$. 

\begin{thm} Define graphs $R$ and $G^{-}_{C_2}$ on $C_2\times 2^{\omega}$ as follows
\[R=\{(x,-x): x\in C_2\times 2^{\omega}\}\] and \[G^{-}_{C_2}=\{(-x,-y): (x,y)\in G_{C_2}\}.\] A graph $G$ on $X$ is generated by a countable family of Borel functions if and only if there is no homomorphism from $(G_{C_2},G^{-}_{C_2}, R)$ to $(=, \not =, G).$
\end{thm}

\begin{proof}
Such a homomorphism amounts to a map ${f: C_2\times 2^\omega}$ so that 
\begin{enumerate}
    \item $f(x)=f(y)$ if $(x,y)\in G_{C_2}$
    \item $(f(x), f(- x))\in G$
    \item If $-x G_{C_2} -y$, then $f(x)\not=f(y)$ (or equivalently if $x G_{C_2} y$ then $f(-x)\not=f(-y)$.)
\end{enumerate}

We will show that such a map exists if and only if there is an equivariant homomorphism from $(G_{C_2}, a_{C_2})$ into the graph $\widetilde G=\{((x,y),(x,y')): (x,y), (x,y')\in G\}$ on $G$.

Suppose $f$ is a map as above. Then define $\tilde f:C_2\times 2^\omega\rightarrow G$ by $\tilde f(x)=(f(x),f(- x)).$ This is well-defined by property $(2)$. If $x,y$ are neighbours in $G_{C_2}$, then by property $(3)$, $f(- x)\not=f(- y)$, and by property $(1)$ $f(x)=f(y)$. So, $\tilde f(x)=(f(x),f(- x))= (f(y),f(- x))$ and $\tilde f(y)=(f(y),f(- y))$ are neighbours in $\widetilde G$. 

Suppose $f$ is an equivariant homomorphism from $(G_{C_2}, a_{C_2})$ to $\widetilde G$. Then define $\tilde f: C_2\times 2^\omega\rightarrow X$ by $\tilde f(x)=x_0$ where $f(x)=(x_0,x_1)$. If $f(x)=(x_0, x_1)\in G$, then $f(- x)=-{(x_0,x_1)}=(x_1,x_0)$, so $\tilde f(- x)=x_1$. This means $\tilde f(x)$ and $\tilde f(- x)$ share an edge in $G$, giving property $(2)$. And, if $x$ and $y$ are neighbours in $G_{C_2}$, then $f(x)$ and $f(y)$ are neighbours in $\widetilde G$, so $x_0=y_0$ and $x_1\not=y_1$; this gives properties $(1)$ and $(3)$ above.
\end{proof}

Note that $(1)$ and $(2)$ above say that $f$ descends to a homomorphism to $G$ from the graph on $(C_2\times 2^\omega)/E_{G_{C_2}}$ with edges $([x],[-x])$.

We can also characterize graphs induced by an action of $\Z^2$. There are two canonical obstructions in this case, one in 2 dimensions and one in 1 dimension. Let $\perp_G$ denote the graph on $G$ where two vertices (meaning edges in $G$) are adjacent if they perpendicular. The proof of Theorem \ref{actions} gives the following:

\begin{prop} If $G$ is locally a square lattice, then $G$ is induced by a free Borel action of $\Z^2$ if and only if $\perp_G$ admits a Borel 2-coloring and any forest of straight lines in $G$ is directable. 
\end{prop} 

We think of $\perp_G$ having a 2-coloring as a 2-dimensional requirement, and having all forests of straight lines be directable as 1-dimensional. These two requirements are independent. %state ben's thing about forests somewhere,

\begin{prop}
There are Borel graphs $G$, $\Gamma$ where $\perp_G$ admits a Borel 2-coloring, and every forest of lines in $\Gamma$ is directable, but where neither is induced by a Borel free action of $\Z^2.$
\end{prop}
\begin{proof}
Let $E=\{e_0,e_1,-e_0,-e_1\}$ be the standard generating set for $\Z^2$, and let $\ip{n,m}$ be the canonical image of $(n,m)\in \Z^2$ in $\aut(\cay(\Z^2, E))$. Suppose ${\Gamma\subseteq\aut(\cay(\Z^2, E))}$ is a subgroup so that $\gamma(0)=0$ and $\ip{\gamma(n,m)}\inv\gamma\ip{n,m}\in \Gamma$ for all $\gamma\in\Gamma$. Note that such a group must be finite. 

Let $\Fr_\Gamma$ be the free part of $[0,1]^{\Z^2}$. The group $\aut(\cay(\Z^2, E))$ acts on $\Fr_\Gamma$ by shifting indices, $\gamma\cdot x=x\circ \gamma\inv$. Let $\tilde S_\Gamma$ be the graph on $\Fr_\Gamma/\Gamma$ defined by 
\[\tilde S_\Gamma:= \{([x],[\ip{a}\cdot x]): x\in \Fr_\Gamma, a\in E\}.\]

First, we show that $\tilde S_\gamma$ is always locally isomorphic to $\cay(\Z^2,E).$ In fact, for any $x\in \Fr_{\gamma}$, $(a,b)\mapsto [\ip{a,b}\cdot x]$ is an isomorphism from $\cay(\Z^2,E)$ to the component of $[x]$. Clearly, this  map is a homomorphism. The relation
\[\gamma\cdot \left(\ip{n,m}\cdot x\right)=\ip{\gamma(n,m)}\cdot \left(\ip{\gamma(n,m)\inv}\gamma\ip{n,m}\cdot x\right)\] says that, for every $n,m\in\Z$ and $\gamma\in \Gamma$, there are unique $(a,b)\in\Z^2$ and $\delta\in \Gamma$ so that $\gamma\cdot (\ip{n,m}\cdot x)=\ip{a,b}\cdot (\delta\cdot x)$. So this map is a bijection. And since $\Gamma$ acts freely and preserves $E$, the same relation implies that non-edges are sent to non-edges.

Second, we show that if $\Gamma$ is nontrivial, then $\tilde S_\Gamma$ is not the Cayley graph of a Borel action of $\Z^2.$ Suppose otherwise and for $x\in\Fr_\Gamma$ consider the bijection $f_x:E\rightarrow E$ defined by
\[f_x(a)=b:\lra [\ip{a}\cdot x]=b\cdot [x].\] The same relation as above says that $f_{\gamma\cdot x}(a)=f_x(\gamma(b)\inv)$ for $\gamma\in \Gamma$. And $f_{\ip{n,m}\cdot x}=f_{x}$ for $n,m\in\Z^2$ since $\ip{n,m}$ commutes with $\ip{a}$ for $a\in E$. 

For some $f:E\rightarrow E$, the set $\{x: f_x=f\}$ has positive measure and is $\Z^2$-invariant. By ergodicity of the shift action, this set has measure one. But then $f_x=f_{\gamma\cdot x}$ for some $x$, which is a contradiction.

All that remains is to find suitable $\Gamma$ and $\Delta$ so that $\perp_{\tilde S_\Gamma}$ has a Borel 2-coloring and every forest of straight lines in $\tilde S_\Delta$ is directable.

Let $\Gamma$ be the group generated by $t_0$ and $t_1$ where $t_0(n,m)=(-n,m)$ and $t_1(n,m)=(n,-m)$. Since $t_i(\pm e_j)\in \{e_j,-e_j\}$ for each $i,j$, we can define a Borel 2-coloring of $\tilde S_\Gamma$ by 
\[f([x],[\pm e_i\cdot x])=i.\]

Let $\Delta$ be the group generated by $s$ where $s(n,m)=(m,n)$. Since $s(-e_i)=-s(e_i)$ for each $i$, we can define an orientation $o$ of $\tilde S_\Delta$ by
\[o=\{([x],[y]): x=e_0\cdot y\mbox{ or }x=e_1\cdot y\}.\] This orientation is balanced when restricted to any straight line, so gives a direction to every forest of straight lines.

\end{proof}

 Miller's $\s L_0$ dichotomy characterizes undirectable forests of lines \cite{directable forests}, and a result of Carroy, Miller, Schrittesser, and Vidyanszky characterizes 2-colorable graphs \cite{2coloring}. In the case of $\perp_G$, we can reflect this to a characterization of $G$, but we need to strengthen the results in \cite{2coloring}. %CItations

\begin{dfn}
Let $s_0=\emptyset$, and for $n>0$ let $s_n=0^{n-1}1$. Define a graph $\s L$ on $2^\omega$ by $\{x,y\}\in  \s L$ if and only if
\[(\exists z\in 2^\omega, n\in\omega) x=s_n\fr 0\fr z, y=s_n\fr 1\fr z. \]

For any graph $G$, let $G_{odd}=\{x,y: d_G(x,g)\mbox{ is odd}\}.$
\end{dfn}

If we let $G_n$ be the graph on $2^n$ gotten by restricting the definition of $G$ to strings of length n, then $G$ is a projective limit of the sequence $\ip{G_n:n\in\omega}$. Each $G_n$ is a path graph, and $G_{n+1}$ is gotten by connecting two endpoints of two copies of $G_{n+1}$.

\begin{thm}[Ess.~{\cite[Theorem 1.1]{2coloring}}] 
If $G$ is a Borel graph and not Borel 2-colorable, then there is a homomorphism of $\s L$ into $G_{odd}$.
\end{thm}

Note that if $G$ is locally a square lattice, then $(\perp_G)_{odd}=\perp_G$. The strengthening we need is a strong form of local injectivity for the homomorphism given by the theorem above.

\begin{thm}
If $G$ is locally a square lattice, and $\perp_G$ does not admit a Borel 2-coloring, then there is a Borel homomorphism from $\s L$ into $\perp_G$ so that, if $x,y$ are connected and $x\not=y$, then $f(x)$ and $f(y)$ do not lie on a straight line.
\end{thm}
\begin{proof} By relativization, we may assume $G$ is $\Delta^1_1$. Suppose $\perp_G$ does not admit a $\Delta^1_1$ 2-coloring. Define \[X:=G\sm\bigcup \{A\in \Delta^1_1: A\mbox{ is }\perp_G\mbox{ independent}\}.\] By reflection, $X$ is the same as $G$ without its $\Sigma^1_1$ independent sets. If $X$ is empty, then we can construct a $\Delta^1_1$ 2-coloring of $\perp_G$. So $X$ is nonempty.

As usual, we will build a homomorphism by constructing trees of conditions forcing the graph relations we want. To ensure we get the kind of injectivity we want, we will make sure our generics force points in the image of our homomorphism to be far apart. 

For edges $e,e'$ let $d_G(e,e')=d_G(e_0,e'_0)$ with respect to the graph metric. Say that $e,e'$ are $n$-spaced if, whenever we have paths $p,p'$ which are perpendicular line segments such that $e_0$ is on $p$ and $e'_0$ is on $p'$, then the length of $p,p'$ are both of length at least $n$.

Fix a strategy for the second player in the Choquet game on $X$ equipped with the Gandy--Harrington topology along with a suitable metric for this topology. We build a tree of reals. We inductively build trees of reals and Gandy--Harrington conditions $\ip{U_\sigma: \sigma\in 2^{<\omega}}, \ip{x_{\sigma}: \sigma\in 2^{<\omega}}, \ip{R_{\sigma,\tau}: \sigma,\tau\in G_n}$, so that
\begin{itemize}
    \item $U_\sigma\subseteq X$, $R_{\sigma,\tau} \subseteq \perp_G$, $x_\sigma\in X$
      \item If $\sigma\subseteq \tau$, then
    \[U_\sigma\subseteq U_\tau\]
    \item For any appropriate increasing sequences $\sigma_n,\tau_n\in 2^{<\omega}$ with ${|\sigma_n|=|\tau_n|=n}$, there are  $\widetilde R_{\sigma_n,\tau_n}$ so that the following is a play of the strong Choquet game following our strategy
    \begin{align*}
        \mbox{I}: & R_{\sigma_n,\tau_n},(x_{\sigma_n},x_{\tau_n}) &\; & R_{\sigma_{n+1},\tau_{n+1}},(x_{\sigma_{n+1}},x_{\tau_{n+1}}) & ... \\
        \mbox{II}: &  \; & \widetilde R_{\sigma_n,\tau_n}& \; & \widetilde R_{\sigma_{n+1},\tau_{n+1}} & \;...
    \end{align*}
    \item For each $n$, there is some $d(m)$ so that, if $\ip{x_\sigma:\sigma\in 2^{n}}$ satisfies $(x_{\sigma},x_{\tau})\in R_{\sigma,\tau}$ for all $(\sigma,\tau)\in G_n$, then the diameter of $\ip{x_\sigma:\sigma\in 2^{n}}$ (with respect to the distance $d_G$ defined above) is at most $d(n)$.
    \item If $(x,y)\in R_{s_n\fr0,s_n\fr 1}$, then $x,y$ are at least $3d(n)$-spaced.
\end{itemize}
This done, we have that, for any $s\in 2^\omega$, $f(s)=\lim_n x_{s\res n}$ is defined an defines a homomorphism as desired.

We do this inductively.  Given $x_\sigma, U_\sigma, R_{\sigma, \tau}$ for $|\sigma|=|\tau|=n$, produce $\widetilde R_{\sigma, \tau}$ and $\widetilde T_\sigma$ using our strategy. Say $f: 2^n\rightarrow X$ is an approximation if, for all $\sigma, \tau\in 2^n$
\begin{enumerate}
    \item $f(\sigma)\in U_\sigma$
    \item if $\sigma G_n\tau$, $f(\sigma)\widetilde R_{\sigma,\tau} f(\tau)$
\end{enumerate} Note that $f(\sigma)=x_\sigma$ is an approximation. Consider
\[V=\{y: (\exists (y_\sigma){\sigma\in 2^n} )\left[ y_{s_n}=y \mbox{ and }f(\sigma)=y_\sigma\mbox{ is an approximation}\right]\}.\]

Since $V$ is $\Sigma^1_1$ and a subset of $X$, $V$ contains perpendicular edges. We need to show it must in fact contain perpendicular edges which are $3d(n)$-spaced. Suppose otherwise. Then, in each component, $V$ is either finite or contained between two parallel straight lines. But then we can refine $V$ to an independent set as follows: 
\begin{itemize} 
\item If $[x]_{\perp_G}\cap V$ is finite, keep $x$ if and only if it is the smallest (with respect to some $\Delta^1_1$ linear order) in this set
\item If $[x]_{\perp_G}\cap V$ is infinite, keep $x$ if and only if $x$ is parallel to the lines bounding this set.
\end{itemize} Since our graph is locally countable, this construction gives a $\Delta^1_1$ set.

Now fix $(y_0, y_1)\in V^2$ which are $3d(n)$-spaced. Set  $x_{s_n\fr 0}=y_0$ and $x_{s_n\fr 1}=y_1$ and for $x_{\sigma\fr i}$ for all other $\sigma,i$ choose witnesses to $y_i\in V.$ Then set $U_\sigma$ to be any small enough ball around $x_\sigma$, and let $R_{\sigma,\tau}$ be appropriate refinements of the $\widetilde R$s which fix the spacing of $y_0$ and $y_1$ to be $m>3d(n)$. Finally, let $d(n+1)=10m$.
\end{proof}

With this in place, we can start characterizing graphs $G$ where $\perp_G$ admits a Borel 2-coloring.

\begin{dfn} If $G$ and $H$ are graphs on $X,Y$ which are locally square lattices, then say $f:X\rightarrow Y$ is a conformal map from $G$ to $H$ if, for any vertices $x,y$ of $G$, $x,y, z$ are on a straight line if and only if $f(x),f(y)$ are on a straight line.\end{dfn} 

Conformal maps preserve perpendicularity, so we can pull a 2-coloring back from $\perp_G$ to $\perp_H$ whenever there is a conformal map from $H$ to $G$.

\begin{prop}
If there is a Borel conformal map from $H$ to $G$ and $\perp_G$ has a Borel 2-coloring, then $\perp_H$ has a Borel 2-coloring.
\end{prop}
\begin{proof}
First note that if $v_1,v_2,v_3,$ and $v_4$ are consecutive corners of a rectangle in a square lattice if and only $v_i$ and $v_j$ are on straight line exactly when $i=j$ (mod 2). So the corners of a rectangle in $H$ are sent to the corners of a rectangle in $G$ by any conformal map.

Now suppose that $f$ is a Borel conformal map from $H$ to $G$ and $g$ is a Borel 2-coloring of $\perp_G$. Then we have a Borel map $\tilde g:H\rightarrow 2$ given by $\tilde g(e_0,e_1)=g(e')$ for any edge $e'$ on the straight line between $f(e_0)$ and $f(e_1)$. If $e\perp_H e'$, then possibly replacing $e$ with $-e$, we can find a rectangle containing $e,e'$. This rectangle is sent by $f$ to a rectangle with consecutive sides containing $f(e_0), f(e_1)$ and $f(e'_0),f(e'_1)$. The coloring $g$ must assign opposite colors to all edges on consecutive sides of a rectangle, so $\tilde g(e)\not=\tilde g(e').$
\end{proof}

Our goal is to show there is a minimal square lattice with respect Borel conformal maps among those where $\perp$ is not Borel 2-colorable.

\begin{dfn}
 Let $X_{\bb S_0}=\{(x,y)\in 2^\omega: d_{\s L}(x,y)\mbox{ is odd}\}/E$ with the quotient Borel structure, where $[(x,y)]_E=\{(x,y),(y,x)\}$. Abbreviate $[(x,y)]_E$ by $[x,y]$. Then $\bb S_0$ is the graph on $X_{\bb S_0}$ defined by
 \[[x,y]\; \bb S_0\; [a,b]:\lra x=a,d_{\s L}(y,b)=2\mbox{, or }y=b,d_{\s L}(x,a)=2.\]
\end{dfn}

Before proving our dichotomy, let us show that $\bb S_0$ is locally a square lattice.

\begin{prop}
$\bb S_0$ is locally a square lattice.
\end{prop}
\begin{proof}
Fix a component $C$ of $\bb S_0$. The corresponding component of $\s L$ is 2-regular and acyclic, so we can fix an isomorphism between it and the usual Cayley graph for $\Z$. So we can represent $C$ as $\{(n,m): n\in 2\Z, m\in 2\Z+1\}$. And we have that $(n,m)$ and $(\ell,k)$ are adjacent if $n=\ell\pm2$ or $m=k\pm 2.$ Thus $(a,b)\mapsto (2a, 2b+1)$ defines an isomorphism between the usual Cayley graph for $\Z^2$ and $C$.
\end{proof}

\begin{thm}\label{axndichotomy}
If a Borel graph $G$ is locally a square lattice, then either $\perp_G$ admits a Borel 2-coloring, or there is a Borel conformal map from $\bb S_0$ to $G$.
\end{thm} 
\begin{proof}
First we show $\bb S_0$ does not admit a conformal map into any graph $G$ where $\perp_G$ is Borel 2-colorable. It suffices to show $\perp_{\bb S_0}$ does not admit a Borel 2-coloring. Note that the straight lines in $\bb S_0$ are the sets $L_a=\{[a,b]: d(a,b)\mbox{ is odd}\}$ for $a\in \s L$. Suppose $f: G\rightarrow 2$ is a Borel 2-coloring of $\perp_{\bb S_0}$. Then for each $a$, the edges in $L_a$ are all assigned the same color by $f$. This defines a Borel function $g:\s L\rightarrow 2$. Suppose $d(a,b)$ is odd, $d(a,a')=d(b,b')=2$. Then $([a,b], [a',b])$ and $([a,b],[a,b'])$ are perpendicular edges in $L_b$ and $L_a$ respectively, so edges in $L_a$ and $L_b$ are assigned different colors by $f$, and $g(a)\not=g(b)$. Thus $g$ is a Borel 2-coloring of $\s L$, which is a contradiction. 

Now suppose $G$ is locally a square lattice and $\perp_G$ does not admit a Borel 2-coloring. So, there is a Borel homomorphism $f$ from $\s L$ to $\perp_G$ which sends connected vertices to edges which are not on a straight line. Define an map $g$ from $\bb S_0$ to $G$ by setting $g([a,b])$ to be the unique point on both the straight line containing $f(a)$ and the straight line containing $f(b)$. This is well-defined since any two perpendicular lines in the same plane must intersect. All that remains is to check this is conformal.

Suppose $[x,y_0]$ and $[x,y_1]$ are two points on a straight line $L_x$. Then $g([x,y_0])$ and $g([x,y_1])$ are on the straight line through $f(x)$. Conversely, suppose $[x_0,x_1]$ and $[y_0,y_1]$ are not a straight line, and $d_{\s L}(x_0,y_0)=0\;(\mbox{mod } 2)$. Then, $x_0\not=y_0$ and $x_1\not=y_1$. So $f(x_0)$, $f(y_0)$, $f(x_1)$, and $f(y_1)$ are all on different straight lines. So, $g([x_0,x_1])$ and $g([y_0,y_1])$ are not on a straight line.  Thus $g$ is conformal.
\end{proof}

Combining this with Miller's $\s L_0$ dichotomy for directable forests of lines, we get the following: %citation

\begin{cor}
If $G$ is a Borel graph, either $G$ is the Schreier graph of a free Borel action of $\Z^2$  or at least one of the following holds:
\begin{enumerate}
    \item $G$ is not locally a square lattice
    \item There is a Borel conformal map from $\bb S_0$ to $G$
    \item There is a Borel betweenness-preserving embedding of $\s L_0$ into some forest of straight lines in $G$
\end{enumerate}
\end{cor}

\end{document}